\documentclass{amsart}
\usepackage{amsmath, amsthm, amssymb}
\usepackage{verbatim}
\usepackage{tikz}
\usetikzlibrary{matrix}

\usepackage[mathscr]{eucal} 
\usepackage{mathrsfs} 
\usepackage{cancel}

\usepackage{enumitem, hyperref}
\makeatletter
\def\namedlabel#1#2{\begingroup
    #2%
    \def\@currentlabel{#2}%
    \phantomsection\label{#1}\endgroup
}

\renewcommand{\d}{\partial}

\newcommand{\wt}[1]{\widetilde{#1}}
\newcommand{\lc}{\left<}
\newcommand{\rc}{\right>}

\newcommand{\eps}{\epsilon}
\newcommand{\veps}{\varepsilon}

\newcommand{\al}{\alpha} 
\newcommand{\ze}{\zeta}

\newcommand{\de}{\delta}

\newcommand{\om}{\omega}

\newcommand{\vka}{\varkappa}

\newcommand{\Om}{\Omega}
\newcommand{\De}{\Delta}

\newcommand{\cP}{\mathcal{P}}

\newcommand{\cH}{\mathcal{H}}
\newcommand{\cL}{\mathcal{L}}

\newcommand{\cO}{\mathcal{O}}

\newcommand{\bB}{\mathbb{B}}
\newcommand{\bK}{\mathbb{K}}
\newcommand{\bP}{\mathbb{P}}
\newcommand{\bR}{\mathbb{R}}

\newcommand{\bC}{\mathbb{C}}

\newcommand{\bN}{\mathbb{N}}
\newcommand{\vpi}{\varpi}

\newcommand{\wh}[1]{\widehat{#1}}
\newcommand{\ov}[1]{\overline{#1}}

\newcommand{\uU}[1]{{\rm U}#1}

\newcommand{\cali}[1]{\mathscr{#1}}

\newcommand{\Cc}{\cali{C}}

\newtheorem{thm}{Theorem}
\newtheorem{prop}[thm]{Proposition}
\newtheorem{lem}[thm]{Lemma}
\newtheorem{cor}[thm]{Corollary}

\theoremstyle{definition}

\newtheorem{defn}[thm]{Definition}
\newtheorem{remark}[thm]{Remark}

\newtheorem{expl}[thm]{Example}
\newtheorem*{rmkx}{Remark}

\numberwithin{thm}{section}
\numberwithin{equation}{section}

\renewcommand{\[}{\begin{equation}}
\renewcommand{\]}{\end{equation}}

\title[Continuity of extremal functions 
]{Characterization of continuity of Siciak-Zaharjuta extremal functions on compact Hermitian manifolds
}

\author{Hyunsoo Ahn} 
\address{Department of Mathematical Sciences, KAIST, 291 Daehak-ro, Yuseong-gu, Daejeon 34141, South Korea}
\email{kakapoolove@kaist.ac.kr}

\begin{document} 


\begin{abstract} 
For a 
compact subset in a compact Hermitian manifold,
we prove that the continuity of the extremal function at a given point in the set
is 
a local property 
and that the continuity of a weighted extremal function follows from the continuities of the extremal function and the weight function. 
These results are generalizations of the results of Nguyen \cite{Ng24} on compact K\"ahler manifolds. 
Moreover, for a compact subset in a compact Hermitian manifold, at the point level and accordingly at the global level, we characterize the continuity of the extremal function via the local \(L\)-regularity, which is equivalent to the weak local \(L\)-regularity. We also show that the \(L\)-regularity of a compact subset in \(\mathbb{C}^n\) at a star center implies the local \(L\)-regularity. Consequently, a convex compact \(L\)-regular subset in \(\mathbb{C}^n\) is locally \(L\)-regular. 
\end{abstract}


\maketitle

\section{Introduction}
{\em Background.}
The Siciak-Zaharjuta extremal function was historically defined first using polynomials by Siciak in 1962 and subsequently characterized using plurisubharmonic functions by Zaharjuta in the 1970s. 
It is considered as an extension of one-dimensional Green function with a logarithmic pole at infinity to higher-dimensional complex spaces using plurisubharmonic (psh for short) functions rather than just subharmonic functions. 


For a subset \(E\) in \(\mathbb{C}^n\), the classical Siciak-Zaharjuta extremal function of \(E\) introduced in \cite{Si62,Siciak81,Za76} is defined on \(\mathbb{C}^n\) by
$$L_{E}(z) := \sup\{ u(z) : u \in \mathcal{L}, \, u|_E \leq 0\},\quad z\in \mathbb{C}^n,$$
where the Lelong class \(\mathcal{L}\) is
$$
\mathcal{L}:=\{u\in PSH(\mathbb{C}^n):
\text{for some constant }c_u,\;
u(z)\leq\max\{0,\log\|z\|\}+c_u,\;z\in\mathbb{C}^n\}.
$$
The regularity of the extremal function of a set gives geometric information about the set itself. For example, $E$ is pluripolar if and only if $L_E^*\equiv\infty$, where \(^*\) denotes the upper semicontinuous regularization, and \(E\) is non-pluripolar if and only if \(L_E^*\in\mathcal{L}\). 

We recall the notions of \(L\)-regularity and local \(L\)-regularity of a subset in \(\mathbb{C}^n\). We denote the open Euclidean ball of center \(a\in\mathbb{C}^n\) and radius \(r\in(0,\infty]\) by \(\mathbb{B}(a,r)\) and its closure by \(\bar{\mathbb{B}}(a,r)\). 
\begin{defn}\label{defn:L-reg}
    Let \(E\) be a subset in $\mathbb{C}^n$ and $a\in \bar{E}$. We say 
    \begin{itemize}
        \item[(i)] $E$ is {$L$-regular at $a$} if $L_E$ is continuous at $a$, or equivalently, \(L_E^*(a)=0\). 
        \item[(ii)] $E$ is {locally $L$-regular at $a$} if $L_{E\cap \bar{\mathbb{B}}(a,r)}$ is continuous at $a$ for each \(r\in(0,\infty]\). 
        \item[(iii)] \(E\) is \(L\)-regular if \(L_E\) is continuous at every point of \(\bar{E}\).
        \item[(iv)] \(E\) is locally \(L\)-regular if it is locally \(L\)-regular at every point of \(\bar{E}\).
    \end{itemize}
\end{defn}
The equivalence in (i) is due to the non-negativity of \(L_E\). 
In (ii), \(L_E\) is locally \(L\)-regular at \(a\) if and only if $L_{E\cap \bar{\mathbb{B}}(a,r)}$ is continuous at $a$ for small \(r>0\). This equivalence follows from the the monotonicity of extremal functions with respect to the set inclusion. By the same reason, \(E\) is locally \(L\)-regular at \(a\) if and only if \(L_{E\cap S}\) is continuous at \(a\) for each neighborhood \(S\subset\mathbb{C}^n\) of \(a\).
About (iii), \(L\)-regularity of a compact set \(F\subset\mathbb{C}^n\) implies the continuity of $L_F$ on $\mathbb{C}^n$ since the extremal function of a compact subset in \(\mathbb{C}^n\) is lower semicontinuous as the supremum of a family of smooth functions (see \cite[Corollary~5.1.3]{Kl91} for the proof) and \(L\)-regularity of \(F\) gives the upper semicontinuity of \(L_F=L_F^*\) from the fact that \(L_F^*=0\) on \(F\). 

Local \(L\)-regularity of a set in \(\mathbb{C}^n\) implies \(L\)-regularity by monotonicity, but the converse is not true. Sadullaev in \cite{Sa16} gives a counterexample \(F_0:=\{z\in\mathbb{C}:|z|=1\}\cup\{0\}\). \(L_{F_0}(z)=L_{\bar{\mathbb{B}}(0,1)}(z)=\log^+|z|\). Thus \(F_0\) is \(L\)-regular, but it is not locally \(L\)-regular at \(0\) as \(L_{F_0\cap \bar{\mathbb{B}}(0,1/2)}(z)=L_{\{0\}}(z)\) is \(0\) at \(z=0\) and \(\infty\) at \(z\neq 0\). However, Cegrell in \cite{Ce82} proved that \(L\)-regularity and local \(L\)-regularity are equivalent for a compact subset in \(\mathbb{R}^n\) as a subset of \(\mathbb{C}^n\).

Local \(L\)-regularity is related to the continuity of {\em weighted} extremal functions.
For a real-valued function $\phi$ on $E$, the weighted extremal function $L_{E,\phi}$ for \((E,\phi)\) is 
$$
L_{E,\phi}(z) := \sup\{u(z): u \in \mathcal{L}, u|_E \leq \phi\},\quad z\in\mathbb{C}^n.
$$
Siciak \cite[Proposition~2.12]{Siciak81} proved that for a compact subset \(F\subset \mathbb{C}^n\) and a bounded lower semicontinuous function \(\psi:F\to \mathbb{R}\), \(L_{F,\psi}\) is also lower semicontinuous as the supremum of a family of smooth functions. Therefore, if \(L_{F,\psi}\) is continuous on \(F\) (continuous at each point in \(F\)), then \(L_{F,\psi}^*|_F=L_{F,\psi}|_F\leq\psi\), \(L_{F,\psi}=L_{F,\psi}^*\) is continuous on \(\mathbb{C}^n\).
Furthermore, \(L_{F,\psi}\) is continuous if \(\psi\) is continuous and \(F\) is locally \(L\)-regular by \cite[Proposition~2.16]{Siciak81}. 

{\em Results.} Throughout this article, $(X, \omega)$ is a 
compact Hermitian manifold of complex dimension $n$. 
A function \(\phi:X\to\mathbb{R}\cup\{-\infty\}\) is called quasi-plurisubharmonic (quasi-psh for short) if \(\phi\not\equiv-\infty\) and \(\phi\) is locally written as the sum of a smooth function and a psh function. 
For the real differential operators \(d:=\partial+\bar{\partial}\) and \(d^c:=i(\bar{\partial}-\partial)\), the family of $\omega$-psh 
functions on \(X\) is defined by 
\[\notag
	PSH(X,\omega) := \{v :X\to \mathbb{R}\cup\{-\infty\}: v
    \text{ is quasi-psh},\,
    \omega + dd^c v \geq 0\}.
\]
We also write \(\omega_v=\omega+dd^c v\) for an \(\omega\)-psh function \(v\). 
If \(\phi:X\to\mathbb{R}\cup\{-\infty\}\) is quasi-psh, then there exists a constant \(c>0\) such that $dd^c\phi+c\omega\geq 0$, since \(X\) is compact and $PSH(X, c_1\omega)\subset PSH(X,c_2\omega)$ for \(0<c_1<c_2<\infty\).

The extremal function \(V_E=V_{\omega;E}\) of a subset $E$ in $X$ is defined by
\[\label{eq:SZ-ext}
	V_E (x)=V_{\omega;E}(x) := \sup\{ v(x): v \in PSH(X,\omega),\; v|_E \leq 0 
    \},\quad x\in X. 
\]
For a weight function \(\phi:E\to\mathbb{R}\), the weighted extremal function of \((E,\phi)\) is
\[\label{eq:w-SZ-ext}
	V_{E,\phi}(x)=V_{\omega;E,\phi} (x) := \sup\{ v(x): v\in PSH(X,\omega), \; v|_E \leq \phi\},\quad x\in X.\]
When the domain of a real-valued function \(\psi\) contains \(E\), we write \(V_{E,\psi|_E}\) as \(V_{E,\psi}\) for convenience.
There is a 
bijection between \(PSH(\mathbb{CP}^n,\omega_{FS})\) and the Lelong class \(\mathcal{L}\) 
(\cite[Chapter~8]{GZ17}). 
When \(\phi\) is a continuous function on a compact subset of \(X\), we can extend $\phi$ to a continuous function on \(X\) with the same supremum norm by Tietze's extension theorem.



\begin{defn}[local \(L\)-regularity in a complex manifold]
\label{defn:loc-L-reg-mfd} 
Let \(Y\) be a complex manifold. Let \(E\) be a subset of \(Y\) and \(a\in\bar{E}\). We say
\begin{itemize}
\item[(i)] \(E\) is locally \(L\)-regular at \(a\) if for 
\textbf{every} \(R\in(0,\infty]\) and \textbf{every} holomorphic coordinate ball \((\Omega,\tau)\) in \(Y\) centered at \(a\) of radius \(R\),
\(\tau(E\cap\tau^{-1}(\bar{\mathbb{B}}(\mathbf{0},R/2)))\) is locally \(L\)-regular at \(\mathbf{0}\),
or equivalently,
\(\tau(E\cap\Omega)\) is locally \(L\)-regular at \(\mathbf{0}\).
Equivalently, we say \(E\) is locally \(L\)-regular at \(a\) if for \textbf{every} holomorphic chart \((O,g)\) in \(Y\) containing \(a\), \(g(E\cap O)\) is locally \(L\)-regular at \(g(a)\).
\item[(ii)] \(E\) is weakly locally \(L\)-regular at \(a\) 
if for \textbf{some} \(R\in(0,\infty]\) and \textbf{some} holomorphic coordinate ball \((\Omega,\tau)\) in \(Y\) centered at \(a\) of radius \(R\),
\(\tau(E\cap\Omega)\) is locally \(L\)-regular at \(\mathbf{0}\).
Equivalently, we say \(E\) is weakly locally \(L\)-regular at \(a\) if for \textbf{some} holomorphic chart \((O,g)\) in \(Y\) containing \(a\), \(g(E\cap O)\) is locally \(L\)-regular at \(g(a)\).
\item[(iii)] \(E\) is locally \(L\)-regular if it is locally \(L\)-regular at every point of \(\bar{E}\).
\item[(iv)] \(E\) is weakly locally \(L\)-regular if it is weakly locally \(L\)-regular at every point of \(\bar{E}\).
\end{itemize}
\end{defn}
In this definition, the first equivalence in (i) is due to the fact that both 
the intersection of \(\tau(E\cap\tau^{-1}(\bar{\mathbb{B}}(\mathbf{0},R/2)))\) with \(\bar{\mathbb{B}}(\mathbf{0},r)\) and the intersection of \(\tau(E\cap\Omega)\) with \(\bar{\mathbb{B}}(\mathbf{0},r)\) are equal to \(\tau(E\cap\tau^{-1}(\bar{\mathbb{B}}(\mathbf{0},r)))\) for each \(r\in(0,R/2]\). 
The reason for the second equivalence in (i) is that, for some \(R_0\in(0,\infty)\), \(\mathbb{B}(g(a),R_0)\subset g(O)\), $$
L_{g(E\cap O)\cap\bar{\mathbb{B}}(g(a),r)}^*(g(a))=L_{g_0(E\cap O_0)\cap\bar{\mathbb{B}}(\mathbf{0},r)}^*(\mathbf{0})=0,\quad r\in (0,R_0/2],$$
for the coordinate ball
\((O_0,g_0):=(g^{-1}(\mathbb{B}(g(a),R_0)),g|_{O_0}-g(a))\)
centered at \(a\). The equivalence in (ii) holds for the same reason as the second equivalence in (i).
\begin{remark}\label{rmk:loc-L-reg-mfd}
In Definition~\ref{defn:loc-L-reg-mfd}, 
let \((\Omega_{\infty},\tau_{\infty})\) be a holomorphic coordinate ball in \(Y\) centered at \(a\) with infinite radius as \(\tau_{\infty}(\Omega_{\infty})=\mathbb{C}^n\). 
Then \((\Omega_2,\tau_2):=(\tau_{\infty}^{-1}(\mathbb{B}(\mathbf{0},2)),\tau_{\infty}|_{\Omega_2})\) is a holomorphic coordinate ball in \(Y\) centered at \(a\) of finite radius, and \(\tau_{\infty}(E\cap\Omega_{\infty})\cap\bar{\mathbb{B}}(\mathbf{0},r)=\tau_2(E\cap\Omega_2)\cap\bar{\mathbb{B}}(\mathbf{0},r)\) for each \(r\in(0,1]\). Therefore, \(E\) is locally \(L\)-regular at \(a\in \bar{E}\) if and only if 
for every holomorphic coordinate ball \((\Omega,\tau)\) in \(Y\) centered at \(a\) of finite radius, \(\tau(E\cap\Omega)\) is locally \(L\)-regular at \(\mathbf{0}\). By the same reason, \(E\) is weakly locally \(L\)-regular at \(a\in\bar{E}\) if and only if for some holomorphic coordinate ball \((\Omega,\tau)\) in \(Y\) centered at \(a\) of finite radius, \(\tau(E\cap\Omega)\) is locally \(L\)-regular at \(\mathbf{0}\).
\end{remark}

We get the following main results.
\begin{thm}\label{thm:characterization-c}  
Let $K$ be a compact subset of a compact Hermitian manifold \(X\) and  $a\in K$. Let \(\bar{B}(a,r)\) be a closed holomorphic coordinate ball in \(X\) with the center \(a\in K\) and the finite radius \(r>0\). Then the following items hold:
\begin{itemize}
\item
[(i)] $V_K$ is continuous at $a$ if and only if  $V_{K \cap \bar{B}(a,r)}$ is continuous at $a$.
\item
[(ii)] If $\phi:X\to\mathbb{R}$ is continuous and $V_{K}$ is continuous, then $V_{K,\phi}$ is continuous.
\end{itemize}
\end{thm}
Theorem~\ref{thm:characterization-c}-(i) is used to prove the following characterizations of the continuity of \(V_K\).
\begin{thm}\label{thm:characterization-c-c}  
Let $K$ be a compact subset of a compact Hermitian manifold \(X\) and  $a\in K$. Then the following items are equivalent. 
\begin{itemize}
    \item[(i)] \(V_K\) continuous at \(a\).
    \item[(ii)] \(K\) is locally \(L\)-regular at \(a\).
    \item[(iii)] \(K\) is weakly locally \(L\)-regular at \(a\).
\end{itemize}
Consequently, the continuity of \(V_K\), local \(L\)-regularity of \(K\) and weak local \(L\)-regularity of \(K\) are equivalent to each other.
\end{thm}
Theorem~\ref{thm:characterization-c}-(i) is a different phenomenon from that of the extremal functions on \(\mathbb{C}^n\), since there are \(L\)-regular compact subsets of \(\mathbb{C}^n\) which are not locally \(L\)-regular. One such example was explained during the explanation of \(L\)-regularity. 
Theorem~\ref{thm:characterization-c}-(i) is also used in the proof of Theorem~\ref{thm:characterization-c}-(ii). 
Theorem~\ref{thm:characterization-c}-(ii) uses Siciak's proof of \cite[Proposition~2.16]{Siciak81} to get the continuity of weighted extremal functions of compact subsets in \(\mathbb{C}^n\) when the sets are locally \(L\)-regular and the weights are continuous. 
N.~C. Nguyen in \cite{Ng24} proved Theorem~\ref{thm:characterization-c} 
for compact K\"ahler manifolds, and we generalize his results to compact Hermitian manifolds. 
Theorem~\ref{thm:characterization-c-c} is indeed a new characterization of the continuity of extremal functions of compact subsets in compact Hermitian manifolds using the new definitions (local \(L\)-regularity and weak local \(L\)-regularity) of a subset in a complex manifold.

\begin{remark}\label{rmk:examples of locLreg}
There are several locally \(L\)-regular compact subsets in \(\mathbb{C}^n\). Closed Euclidean balls with finite radii are locally \(L\)-regular by \cite[Example~5.1.1]{Kl91}. The image of a locally \(L\)-regular compact subset \(F\) in \(\mathbb{C}^n\) by an invertible affine automorphism \(\Phi\) of \(\mathbb{C}^n\) is also locally \(L\)-regular by the relation 
$$L_{\Phi(F)}=L_{F}\circ\Phi^{-1}.$$
Accordingly, the real cube \([0,1]^n\) is locally \(L\)-regular by \cite[Corollary~5.4.5]{Kl91}. 
For other examples, see \cite[Example~4.1, Example~4.2]{Ng24}.
\end{remark}

The \(L\)-regularity at a point in a general compact subset in \(\mathbb{C}^n\) does not imply the local \(L\)-regularity at the point. However, if the point is a star center of the compact subset, then we prove that both regularities are equivalent at the point. In particular, convex \(L\)-regular compact subset in \(\mathbb{C}^n\) is locally \(L\)-regular. 
\begin{thm}\label{thm:L-reg at starcenter implies loc L-reg}
Let \(F\) be a compact subset of \(\mathbb{C}^n\). If \(F\) is \(L\)-regular at \(b\in F\) and \(b\) is a star center of \(F\), i.e., \(tF+(1-t)b\subset F\) for each \(t\in[0,1]\),  then \(F\) is locally \(L\)-regular at \(b\). Consequently, if \(F\) is \(L\)-regular and convex, then \(F\) is locally \(L\)-regular. 
\end{thm}

\bigskip
{\em Organization.} In Section~\ref{sec:c}, we give basic properties of extremal functions on a compact Hermitian manifold \(X\) and prove Theorem~\ref{thm:characterization-c}, Theorem~\ref{thm:characterization-c-c} and Theorem~\ref{thm:L-reg at starcenter implies loc L-reg}. 
Section~\ref{sec:appendix} is an appendix including the lemma explaining the relation between extremal functions on \(X\) and locally defined weighted relative extremal functions in \(\mathbb{C}^n\), which is used for the proof of Theorem~\ref{thm:characterization-c-c}. 

\bigskip
{\em Acknowledgement.} I would like to thank my advisor Ngoc Cuong Nguyen with my deep gratitude for helping me write this article. This work was supported by the National Research Foundation of Korea (NRF) funded by the Korean government (MSIT) RS-2026-25470686.

\section{Continuity of extremal functions}\label{sec:c}
Throughout this article, \((X,\omega)\) is a compact Hermitian manifold of complex dimension \(n\). We impose no additional assumption on the Hermitian metric \(\omega\).

\subsection{Basic properties of extremal functions on a compact Hermitian manifold} The basic properties introduced in this subsection are the analogues of the extremal functions on \(\mathbb{C}^n\).
Following Guedj and Zeriahi \cite[Definition~9.20]{GZ17}, we define the Alexander-Taylor capacity for subsets of \(X\). Note that for subsets \(E\) and \(O\) of \(X\) (resp. of \(\mathbb{C}^n\)), if \(O\) is open in \(X\) (resp. in \(\mathbb{C}^n\)), then \(\sup_O V_E=\sup_O V_E^*\) (resp. \(\sup_O L_E=\sup_O L_E^*\)). 
\begin{defn}
Let \(E\) be a subset of \(X\). Alexander-Taylor capacity \(T_{\omega}(E)\) of \(E\) is defined by \(T_{\omega}(E):=\exp(-\sup_X V_E)\).
\end{defn}
\begin{prop}\label{prop:ATcap}
Let \(E\) be a subset of \(X\). Then \(T_{\omega}(E)\) is equal to
$$T_{\omega}'(E):=\inf\{e^{\sup_{E}\psi}:\psi\in PSH(X,\omega),\,\sup_X \psi=0\}.$$
\end{prop}
\begin{proof}
Suppose \(\psi\in PSH(X,\omega)\) and \(\sup_X\psi=0\). Then 
\(\psi-sup_E \psi\leq 0\) on \(E\),
$$\sup_X(\psi-\sup_E \psi)=-\sup_E \psi
\leq \sup_X V_E,\quad \sup_E \psi\geq -\sup_X V_E.$$ 
Accordingly, 
\(T_{\omega}(E)\leq T'_{\omega}(E)\).

Conversely, 
for \(v\in PSH(X,\omega)\) with \(v|_E\leq0\),
we have \(\sup_X(v-\sup_X v)=0\),
$$\sup_E v-\sup_X v
=\sup_E(v-\sup_X v)
\geq \inf_{\psi\in PSH(X,\omega),\sup_X\psi=0} \sup_E \psi=\log{T_{\omega}'(E)}.$$
It follows that 
$$0\geq\sup_{v\in PSH(X,\omega),v|_E\leq 0}\sup_{E}v
\geq \log{T_{\omega}'(E)}+\sup_{v\in PSH(X,\omega),v|_E\leq 0}\sup_X v.$$
Since \(\sup_v\sup_X v\) is equal to \(\sup_X V_E\),
we get the converse inequality.
\end{proof}

The following proposition by Vu \cite[Lemma~2.2, 2.7]{Vu19}
gives \(\omega\)-psh functions. For an open subset \(U\) of \(X\), $PSH(U,\omega)$ is defined by $$PSH(U,\omega):=\{v:U\to\mathbb{R}\cup\{-\infty\}:v\text{ is quasi-psh},\,\omega+dd^cv\geq0\}.$$
\begin{prop}\label{prop:envelope}
(i) If \(U_1, U_2\subset X\) are open with \(\overline{U_1}\subset U_2\), \(v_1\in PSH(U_1,\omega)\) and \(v_2\in PSH(U_2,\omega)\) with \(\limsup_{U_1\ni y\to x}v_1(y)\leq v_2(x)\) for each \(x\in\partial U_1\), then \(v:U_2\to[-\infty,\infty)\)
defined as \(\max\{v_1,v_2\}\) on \(U_1\) and \(v_2\) on \(U_2\setminus U_1\) is in \(PSH(U_2,\omega)\).\\
(ii) If \(\{v_{\beta}\}_{\beta\in I}\subset PSH(X,\omega)\) are  uniformly bounded above, then \((\sup_{\beta\in I} v_{\beta})^*\in PSH(X,\omega)\).
\end{prop}

\begin{proof}
These are proved using \cite[Lemma~2.1]{Vu19}, which is the characterization of an \(\eta\)-psh function \(\Phi\) on an open set \(W\) of \(\mathbb{C}^n\) for a continuous real \((1,1)\)-form \(\eta\) on \(W\), as an 
upper semicontinuous function satisfying the following integral inequality
$$
\Phi(x)\leq\frac{1}{2\pi}\int_0^{2\pi}\Phi(x+\epsilon e^{i\theta}v)d\theta+\frac{1}{2\pi}\int_0^{\epsilon}\frac{dt}{t}\int_{\{s\in\mathbb{C}:|s|\leq t\}}\eta(x+sv)$$
for each \((x,\epsilon,v)\in W\times(0,\infty)\times(\mathbb{C}^n\setminus\{{\mathbf{0}}\})\) of \(\bar{\mathbb{B}}(x,\epsilon\|v\|)\subset W\). 
\end{proof}

Pluripolar sets and \(\omega\)-pluripolar sets are also defined on \(X\). 
\begin{defn}
Let \(E\) be a subset of \(X\). \(E\) is called pluripolar if for each \(p\in E\), there exists open \(U\ni p\) and \(\phi\in PSH(U)\) such that \(\phi\not\equiv-\infty\) and \(E\cap U\subset \{\phi=-\infty\}\). \(E\) is called \(\omega\)-pluripolar if \(E\subset\{v=-\infty\}\) for some \(v\in PSH(X,\omega)\).
\end{defn}

We have the following equivalence due to Vu in \cite[Theorem~1.1]{Vu19}.
\begin{prop}\label{lppgpp}
\(E\subset X\) is pluripolar if and only if it is \(\omega\)-pluripolar.
\end{prop}
Accordingly, a countable union of pluripolar set is pluripolar. The reason is the following. If \(E_j\subset\{\varphi_j=-\infty\}\) for some \(\varphi_j\in PSH(X,\omega)\) with \(\sup_X \varphi_j=0\), \(\varphi:=\sum_{j=1}^{\infty}\frac{\varphi_j}{2j^2}\) is the pointwise limit of decreasing sequence of upper semicontinuous functions \((\sum_{j=1}^{k}\frac{\varphi_j}{2j^2})_{k\in\mathbb{N}}\)
so \(\varphi\) is upper semicontinuous. \(\varphi\) is also the 
\(L^1\) limit of \(\sum_{j=1}^{k}\frac{\varphi_j}{2j^2}\in PSH(X,\omega)\) as \(\sum_{j=1}^{\infty}\frac{1}{2j^2}=\frac{\pi^2}{12}\leq1\) and \(\{\varphi_j\}_{j\in\mathbb{N}}\) is bounded in \(L^1(X,\omega^n)\) as it is a well-known fact that \(\{v\in PSH(X,\omega):\sup_Xv=0\}\) is bounded in \(L^1(X,\omega^n)\), 
so \(\varphi\in PSH(X,\omega)\) and \(\cup_j E_j\subset\{\varphi=-\infty\}\).

Characterization of pluripolar sets using their extremal functions is possible. 
\begin{prop}\label{prop:pluripolarextremal}
Let \(E\) be a subset of \(X\). 
\begin{itemize}
\item[(i)] E is pluripolar \(\Leftrightarrow\) \(\sup_X V_E^*=\infty\) \(\Leftrightarrow\) \(V_E^*\equiv\infty\).
\item[(ii)] If E is not pluripolar, then \(V_E^*\in PSH(X,\omega)\) and \(V_E^*|_{int(E)}\equiv0\).
\end{itemize}
\end{prop}

\begin{proof}
(i) By \cite[Lemma~2.6]{Vu19} and Proposition~\ref{prop:ATcap}, \(E\) is pluripolar if and only if \(\sup_X V_E^*=\infty\). \(\sup_X V_E^*=\infty\) implies \(V_E^*\equiv\infty\) by \cite[Theorem~9.17]{GZ17} and Proposition~\ref{lppgpp}.

(ii) By (i) and Proposition~\ref{prop:envelope}, \(V_E^*\in PSH(X,\omega)\). \(V_E=0\) on \(E\) gives the rest. 
\end{proof}


We also have the following properties as in 
\cite[Proposition~9.19]{GZ17}.
\begin{prop}\label{prop:elementary} 
Let \(E,F,P\) be subsets of \(X\). 
\begin{enumerate}
\item[(i)]
$E\subset F$ implies $V_F \leq V_E$.
\item[(ii)]
If $E$ is open, $V_E = V^*_{E}$.
\item[(iii)]
If $P$ is pluripolar, $V_{E\cup P}^* = V_E^*$.
\item[(iv)]
If $E_j\subset X$ is an increasing sequence and $E$ is the limit, 
$\lim_{j\to \infty}V_{E_j}^* = V_E^*.$
\item[(v)]
If compact $K_j\subset X$ is a decreasing sequence and $K$ is the limit,\\
$\lim_{j\to \infty} V_{K_j} = V_K$ and $\lim_{j\to \infty} V_{K_j}^*  = V_K^*$ a.e..
\end{enumerate}
\end{prop}

The lower semicontinuity of the extremal functions on \(X\) for compact sets holds.  
\begin{lem} \label{lem:lower-semicontinuity} If \(K\) is a compact subset of \(X\), then \(V_K\) is lower semicontinuous. 
\end{lem}
\begin{proof} Let \(v\) be a competitor for \(V_K\) as $v \in PSH(X, \omega)$, $v|_K\leq 0$. 
\cite{BK07} guarantees a sequence $v_j \in PSH(X,\omega) \cap C^\infty(X)$ decreasing to $v$. Fix any \(\delta>0\).
For each locally defined bounded smooth function \(\rho\) with \(dd^c\rho\geq \omega\), smooth psh \(v_j+\rho\) decrease to psh \(v+\rho\) pointwisely, so the convergence holds in \(L^1_{loc}\)-topology and as distributions on the domain of \(\rho\) by Hartogs lemma \cite[Theorem~1.46]{GZ17}. 
\cite[Theorem~1.46]{GZ17} also tells that on each compact subset \(K_\rho\) of the intersection of \(K\) and the domain of \(\rho\), since \(\rho\) is lower semicontinuous on \(K_\rho\), for some \(N(K_\rho)\in\mathbb{N}\), 
$$v_j=(v_j+\rho)-\rho\leq ((v+\rho)-\rho))+\delta=v+\delta\quad \text{on }K_{\rho}, \quad j\geq N(K_{\rho}).$$
\(K\) can be covered by finitely many such \(K_{\rho}\), so \(v_j|_K\leq (v+\delta)|_K\leq \delta\) for large \(j\).
Thus \(V_K=\sup\{\varphi\in PSH(X,\omega)\cap C^{\infty}(X):\varphi|_K\leq0\}\) is lower semicontinuous.
\end{proof}
\begin{cor}\label{cor:semicontinuity} If \(K\) is a compact subset of \(X\), then $V_K$ is continuous if and only if $V_K^*=0$ on \(K\).
\end{cor}
\begin{proof}
\(V_K\) is lower semicontinuous
by Lemma~\ref{lem:lower-semicontinuity}.
If \(V_K^* = 0\) on \(K\), then  
\(V_K^*\in PSH(X,\omega)\) by Proposition~\ref{prop:pluripolarextremal}, \(V_K^*\leq V_K\), \(V_K=V_K^*\), \(V_K\) is upper semicontinuous and thus continuous. Conversely, if \(V_{K}\) is continuous, then \(V_{K}^*= V_{K}\). Accordingly, \(0\leq V_{K}^*=V_{K}\leq0\) on \(K\).
\end{proof}

\begin{remark}\label{rmk:regularity-c} Corollary~~\ref{cor:semicontinuity} implies that the 
continuity of $V_K$ for compact \(K\) does not depend on the Hermitian metric $\omega$ since for another hermitian metric $\omega'$, $\omega' \leq c \omega$ for some constant \(c>0\), $V_{\omega';K}^* \leq c V_{\omega;K}^*=cV_{K}$. 
\end{remark}

The (zero-one) relative extremal function for a set $E$ in $X$ is defined by
\[\label{eq:r-ext}
	h_E (z) := \sup\left\{ v(z): v\in PSH(X,\omega),\; v\leq 1, 
    \;v|_E\leq 0\right\}.
\]
By Proposition~\ref{prop:envelope}, 
$h_E^* \in PSH(X,\omega)$. 
Lemma~\ref{lem:BTcap_globalandlocal} proves that \(E\) is pluripolar if and only if \(h_E^*\equiv 1\) on the compact Hermitian manifold \(X\).
For a non-pluripolar compact set $K\subset X$, denote $M_K := \sup_{X}V_K^*=\sup_{X}V_K<\infty$. Then we have 
\[\label{eq:zero-one-function}
	V_K^* \leq M_K h_K^*.\]


The global Bedford-Taylor capacity is comparable (bi-Lipschitz equivalent) to the local Bedford-Taylor capacity. 
The proof is similar to the compact K\"ahler case as in \cite[page 52-53]{Ko05}, \cite[Proposition~9.8]{GZ17}. This comparability, Proposition~\ref{lppgpp}, \cite[Corollary~4.36]{GZ17} and \cite[Theorem~4.40]{GZ17} together characterize Borel pluripolar sets and pluripolar sets via the global Bedford-Taylor capacity and the outer global Bedford-Taylor capacity respectively. 
Characterization of a pluripolar set using its relative extremal function is also possible.

Let \(\Omega\) be a smoothly bounded strictly pseudoconvex domain in \(\mathbb{C}^n\). 
Bedford and Taylor in \cite{BT82} first studied the relative capacity of a Borel subset \(E\) 
in \(\Omega\), which is defined by 
$$Cap(E,\Omega):=\sup\{\int_{E}(dd^cu)^n:u\in PSH(\Omega),\,0\leq u\leq 1\}.$$
The global Bedford-Taylor capacity of a Borel subset \(E\) in \(X\) is defined as 
$$Cap_{\omega}(E):=\sup\{\int_{E} (\omega+dd^cv)^n:v\in PSH(X,\omega),\,0\leq v\leq1\}.$$ 
Dinew in \cite{De16} pointed out that smooth \(\omega>0\) and compact \(X\) guarantees a constant \(C_{\omega}>0\) satisfying \[
-C_{\omega}\omega^2\leq ni\partial\bar{\partial}\omega\leq C_{\omega}\omega^2,
\quad
-C_{\omega}\omega^3\leq n^2i\partial\omega\wedge\bar{\partial}\omega\leq C_{\omega}\omega^3.\label{eq:condition(B)}
\]
Using \eqref{eq:condition(B)} and induction,
\cite[Proposition~2.3]{DK12} verified that \(Cap_{\omega}(X)\) is finite. 

The outer global Bedford-Taylor capacity of a subset \(E\) in \(X\) is defined as 
$$Cap_{\omega}^*(E):=\inf\{Cap_{\omega}(O):O\;\text{open, }E\subset O\subset X\}.$$

\begin{lem}\label{lem:BTcap_globalandlocal} Let \(\rho_j\), \(1\leq j\leq N<\infty\), be finitely many smooth strictly psh functions 
on some holomorphic charts of \(X\) properly containing the closures of \(U_j=\{\rho_j<0\}\neq\emptyset\) respectively. Assume for some constant \(\delta>0\), \(U_{j,\delta}=\{\rho_j<-\delta\}\) cover \(X\).
Then there exist some constants \(C_j=C_j(\rho_j,X,\omega)>1\),  \(C_j'=C_j'(\rho_j,X,\omega,\delta)>1\) such that 
for each \(1\leq j_0\leq N\) and Borel subset \(E\subset X\),
\[\begin{aligned} 
\label{eq:omcap-under-relcap}
Cap_{\omega}(E\cap U_{j_0,\delta})&\leq C_{j_0} Cap(E\cap U_{j_0,\delta},U_{j_0}),\\
Cap_{\omega}(E)
\leq\sum_{j=1}^NCap_{\omega}({E\cap U_{j,\delta}})
&\leq(\max_{1\leq j\leq N}C_j)\sum_{j=1}^NCap({E\cap U_{j,\delta}},U_j),\end{aligned}\]
\[\label{eq:relcap-under-omcap}
\begin{aligned}
Cap(E\cap U_{j_0,\delta},U_{j_0})
&\leq C_{j_0}' Cap_{\omega}(E\cap U_{j_0,\delta})
\leq C_{j_0}' Cap_{\omega}(E),\\
\sum_{j=1}^N Cap(E\cap U_{j,\delta},U_j)
&\leq (\sum_{j=1}^N C_j') Cap_{\omega}(E).
\end{aligned}\]
Consequently, a Borel subset \(P_1\subset X\) is pluripolar if and only if \(Cap_{\omega}(P_1)=0\), and a subset \(P_2\subset X\) is pluripolar if and only if \(Cap_{\omega}^*(P_2)=0\) if and only if \(h_{P_2}^*\equiv 1\).
\end{lem}
\begin{proof}
We here give the proof of the characterization of pluripolar sets using their relative extremal functions.
Suppose \(P_2\) is pluripolar. Let \(v\in PSH(X,\omega)\) with \(v\leq1\) be a competitor for \(h_{\emptyset}\equiv1\). Since \(P_2\) is pluripolar, \(P_2\subset\{\varphi=-\infty\}\) for some \(1\geq\varphi\in PSH(X,\omega)\). For each \(\epsilon\in(0,1)\), \((1-\epsilon)v+\epsilon\varphi\leq h_{P_2}\), 
$$\lim_{\epsilon\to0^+}((1-\epsilon)v+\epsilon\varphi)|_{\{\varphi>-\infty\}}=v|_{\{\varphi>-\infty\}}\leq h_{P_2}|_{\{\varphi>-\infty\}},$$
where \(\{\varphi\equiv-\infty\}\) has zero volume. Thus \(v\leq h_{P_2}\leq h_{P_2}^*\) a.e. on \(X\) and \(v+\rho\leq h_{P_2}^*+\rho\) a.e. on the domain of each locally defined smooth function \(\rho\) with \(dd^c\rho\geq\omega\). 
The functions u\(v+\rho\) and \(h_{P_2}^*+\rho\) are plurisubharmonic on the domain of \(\rho\), so \(v+\rho\leq h_{P_2}^*+\rho\) on the domain of \(\rho\). 
Then, \(v\leq h_{P_2}^*\). 
Accordingly, \(1\equiv h_{\emptyset}\leq h_{P_2}^*\leq 1\) and we have \(h_{P_2}^*\equiv 1\).

For the converse part, suppose \(h_{P_2}^*\equiv 1\). By \cite[Lemma~6.4]{KN25} with \(m=n\), 
$$0\leq Cap_{\omega}^*(P_2)\leq C\sum_{s=0}^n\int_X-(h_{P_2}^*-1)(dd^ch_{P_2}^*+\omega)^s\wedge\omega^{n-s}$$
holds, where 
\(0\leq C\) is a uniform constant depending only on \((X,\omega
)\). 
Since \(h_{P_2}^*\equiv 1\), we have \(Cap_{\omega}^*(P_2)=0\). Therefore, \(P_2\) is pluripolar.
\end{proof}

For a non-pluripolar compact subset \(K\subset X\), Vu \cite[Proposition~2.9]{Vu19} obtained the following estimate of \(M_K\) in \eqref{eq:zero-one-function}
using the global Bedford-Taylor capacity. 
\begin{lem}\label{lem:BTcapATcap}
There exists a uniform positive constant \(A=A(X,\omega)\) depending only on \((X,\omega)\) such that, for every non-pluripolar compact subset \(K\) in \(X\), 
\[\label{eq:cap-comparison}0<\frac{(\int_X(dd^cV_K^*+\omega)^n)^{\frac{1}{n}}}{[Cap_{\omega}(K)]^{\frac{1}{n}}}\leq\max\{1,\sup_X V_K\},\quad \sup_X V_K\leq\frac{A}{Cap_{\omega}(K)}.\]
\end{lem}
In \eqref{eq:cap-comparison}, since the Hermitian metric \(\omega\) is non-collapsing, \(\int_X(dd^cV_K^*+\omega)^n>0\). 
To prove the inequality \(\sup_X V_K\leq A/Cap_{\omega}(K)\) in \eqref{eq:cap-comparison}, Vu only used the fact that the family \(\{v\in PSH(X,\omega):\sup_X v=0\}\) is bounded in \(L^1(X,\omega^n)\).

For subsets \(E_1\subset E_2\) in \(X\) and real-valued functions \(\phi_1\leq \phi_2\) on \(X\), there is also monotonicity for weighted extremal functions as
\[ \label{eq:monotonicity}
	V_{E_2, \phi_2} \leq V_{E_1,\phi_1} \leq V_{E_1, \phi_2}.
\]
The following lemma tells other basic properties of weighted extremal functions. 
\begin{lem}\label{lem:weight-vs-unweight} 
For a compact subset \(K\) in \(X\) and a bounded function \(\phi:X\to\mathbb{R}\),
\begin{itemize}
\item[(i)] $V_K^* + \inf_K \phi \leq V_{K,\phi}^* \leq V_K^* + \sup_K \phi.$
\item[(ii)] for the current $\theta := \omega + dd^c \phi$
of bidegree \((1,1)\) and
the function $V_{\theta;K} := \sup\{v \in PSH(X, \theta): v|_K \leq 0\}$ where \(PSH(X,\theta):=\{v:X\to\mathbb{R}\cup\{-\infty\}:
v\text{ is quasi-psh},\,\theta+dd^cv\geq0\}\), $V_{K, \phi}=V_{\omega;K,\phi} = V_{\theta;K} + \phi.$
\item[(iii)] if \(\phi\) is additionally assumed to be lower semicontinuous, then $V_{K,\phi}$ is continuous if and only if $V_{K,\phi}^*\leq\phi$ on \(K\).
\end{itemize}
\end{lem}

\begin{proof} 
(i) and (ii) follow from the definitions of extremal functions. For (iii), \(V_{K,\phi}\) is lower semicontinuous as the supremum of a family of smooth functions 
by the proof of Lemma~\ref{lem:lower-semicontinuity}. Thus \(V_{K,\phi}\) is continuous if and only if \(V_{K,\phi}\) is upper semicontinuous if and only if $V_{K,\phi}^*\leq\phi$ on \(K\). 
\end{proof}

\subsection{Continuity of extremal functions on a compact Hermitian manifold}
\begin{proof}[Proof of Theorem~\ref{thm:characterization-c}]
(i) Let \(E\) be a subset of \(X\) and \(b\in E\).  
If \(V_E\) is continuous at \(b\), then \(V_E^*(b)=V_E(b)=0\). Conversely, \(V_E^*(b)=0\) implies $$\limsup_{z\to b}V_E(z)=0=\liminf_{z\to b}V_E(z), \quad\lim_{z\to b}V_E(z)=0=V_E(b).$$
Therefore, \(V_E\) is continuous at \(b\) if and only if \(V_E^*(b)=0\). 

If \(V_{K\cap\bar{B}(a,r)}\) is continuous at \(a\), then 
\(V_K\) is continuous at \(a\) since 
$$0\leq V_K^*(a)\leq V_{K\cap\bar{B}(a,r)}^*(a)=0, \quad V_K^*(a)=0.$$

Conversely, assume \(V_K\) is continuous at \(a\). 
It is enough to show $V_{K\cap \bar{B}(a,r)}^*(a)=0$. 
Take \(\psi\in C^{\infty}(\mathbb{R},[0,1])\) supported on \([-1,1]\) with \(\psi(0)=1\) and the 
chart \((U,f)\) of \(X\) with \(U\supset\bar{B}(a,r)=f^{-1}(\bar{\mathbb{B}}(f(a),r))\). They induce 
\(\chi_r\in C^{\infty}(X,[0,1])\) defined as
$$\chi_r(x):=\begin{cases}1-\psi(\|f(x)-f(a)\|/r),\quad x\in U,\\
1,\quad x\in X\setminus \bar{B}(a,r).\end{cases}$$ 
$\chi_r(a)=0$ and $\chi_r \equiv 1$ 
on \(X\setminus \bar{B}(a,r)\). Its \(C^1\)-norm and \(C^2\)-norm are bounded by
\[\label{eq:cutoff-fct}
	\|\chi_r\|_{C^1(X)} \leq c_1 /r, \quad \|\chi_r\|_{C^2(X)} \leq c_2 / r^2.
\]
Here, $c_1>0$ and $c_2>0$ are constants independent of $a$ and $r$,  
but they 
depend on the choice of the holomorphic coordinate chart \((U,f)\).
Accordingly, since \(\omega>0\), \(-dd^c\chi_r\geq-(c_2'/r^2)\omega\) for some constant \(c_2'\geq1\) depending only on 
\((U,f)\). Then for \(\varepsilon_r:=\min\{{r^2}/{c_2'},{1}\}/2\), the function \(-\varepsilon_r \chi_r\) belongs to \(PSH(X, \omega/2)\).

Let \(v\in PSH(X,\omega)\), \(v\leq1\) and \(v\leq 0\) on \(K\cap\bar{B}(a,r)\). By the choice of \(\chi_r\) and \(\varepsilon_r\), \(\varepsilon_r(v-\chi_r)\leq0\) on \(K\) and \(\varepsilon_r+1/2\leq 1\). These imply \(\varepsilon_r(v-\chi_r)\leq V_K\). Accordingly, 
\[\label{eq:basic-ineq}
\varepsilon_r  h^*_{K\cap  \bar{B}(a,r)} - \varepsilon_r \chi_r  \leq V_{K}^*.\]
\eqref{eq:basic-ineq} and $V_K^*(a) =0=\chi_r(a)$ 
give $h_{K\cap  \bar{B}(a,r)}^*(a) =0$, $h_{K\cap  \bar{B}(a,r)}^*\not\equiv1$. 
By Lemma~\ref{lem:BTcap_globalandlocal}, $K \cap \bar{B}(a,r)$ is not pluripolar and $Cap_{\omega}(K \cap \bar{B}(a,r))>0$. By \eqref{eq:cap-comparison}, 
$$\sup_X V_{K\cap  \bar{B}(a,r)}^*\leq A/{Cap_{\omega}(K \cap \bar{B}(a,r))} < \infty.$$
By \eqref{eq:zero-one-function}, 
$V_{K\cap \bar{B}(a,r)}^*(a) \leq \sup_X V_{K\cap  \bar{B}(a,r)}^* h^*_{K\cap \bar{B}(a,r)}(a)$. Therefore, $V_{K\cap \bar{B}(a,r)}^*(a)=0.$ 

(ii) By Lemma~\ref{lem:weight-vs-unweight}, to prove the continuity of \(V_{K,\phi}\), it is enough to show $V_{K,\phi}^*\leq \phi$ on \(K\). Let $b\in K$ and $\varepsilon'>0$. We have $\phi|_{\bar{B}(b,r)} \leq \phi(b) + \varepsilon'$ for small $r>0$. By \eqref{eq:monotonicity},
$$V_{K,\phi}\leq V_{K\cap \bar{B}(b,r),\phi}\leq V_{K\cap \bar{B}(b,r), \phi(b)+ \varepsilon'} = \phi(b)+\varepsilon' + V_{K\cap \bar{B}(b,r)}. 
$$
Since \(V_K\) is assumed to be continuous,
\(V_{K\cap \bar{B}(b,r)}^*(b)=0\) by Theorem~\ref{thm:characterization-c}-(i). Thus, $V_{K,\phi}^*(b) \leq \phi(b)+\veps'$. Letting $\varepsilon'\to0^+$ gives $V_{K,\phi}^* (b)\leq \phi (b)$. Since 
\(b\in K\) was arbitrary, $V_{K,\phi}^*\leq \phi$ on \(K\).
\end{proof}

We have a partial converse of Theorem~\ref{thm:characterization-c}-(ii) with appropriately scaled quasi-psh 
continuous weight functions on compact Hermitian manifolds. Nguyen in \cite{Ng24} proved this partial converse on compact K\"ahler manifolds.
\begin{cor}\label{cor:loc-w-c} Let \(K\) be a compact subset of \(X\) and $\phi$ be a quasi-psh continuous real-valued function on \(X\) with constants \(c>0\) and \(c'>0\) such that \(c\omega+dd^c\phi\geq c'\omega\) as currents. If \(V_{K,\phi/c}\) is continuous. Then $V_K$ is continuous. In particular, if \(\psi:X\to\mathbb{R}\) is smooth with a constant \(c''>0\) such that \(c''\omega+dd^c\psi>0\) and \(V_{K,\psi/c''}\) is continuous, then \(V_K\) is continuous.
\end{cor}

\begin{proof} Let $\theta:= c\omega + dd^c \phi \geq c'\omega$ 
as currents. 
Define $Cap_\theta(\cdot)
$ as \(Cap_{\omega}(\cdot)\), using \(\theta\) instead of \(\omega\). Let \(\|\phi\|_X\) be the supremum norm of \(\phi\) on \(X\). Finiteness of \(Cap_{cc'\omega/(1+2\|\phi\|_X)}(X)\) implies finiteness of \(Cap_{\theta}(X)\).
For any locally defined smooth potential functions \(\rho_1\) and \(\rho_2\) satisfying \(dd^c\rho_1\leq \omega\leq dd^c\rho_2\), we have \(dd^c (c'\rho_1)\leq\theta\leq dd^c(c\rho_2+\phi)\) as currents. Accordingly, not only $Cap_\omega(\cdot)$ but also $Cap_\theta(\cdot)$ 
is bi-Lipschitz equivalent to \(\sum_{j=1}^NCap(\cdot\cap U_{j,\delta},U_j)\) by the same proof used for Lemma~\ref{lem:BTcap_globalandlocal}. The continuous local potential \(c\rho_2+\phi\) with \(\theta\leq dd^c(c\rho_2+\phi)\) and the boundedness of \(\phi\) 
implies that \(\theta^n(X)\) is finite and 
\(\{\psi\in PSH(X,\theta):\sup_X\psi=0\}\) is bounded in \(L^1(X,\theta^n)\), by the same reasoning for a smooth Hermitian case. 
Then by the same proof used for the second inequality in \eqref{eq:cap-comparison}, there exists a uniform  constant \(A_\theta>0\) such that, for each non-pluripolar compact set \(F_1\) in \(X\),
$$\sup_X V_{\theta;F_1}\leq \frac{A_\theta}{Cap_{\theta}(F_1)}.$$

By Lemma~\ref{lem:BTcap_globalandlocal} and the comparability of \(Cap_{\omega}(\cdot)\) and \(Cap_\theta(\cdot)\), 
for each compact set \(F_2\) in \(X\), we know that 
\(F_2\) is non-pluripolar if and only if \(Cap_{\theta}(F_2)\in(0,\infty)\).
Therefore, the proof of Theorem~\ref{thm:characterization-c}-(i) is true in $PSH(X,\theta)$. In other words, for each \(a\in K\) and 
for each closed holomorphic coordinate ball \(\bar{B}(a,r)\) in \(X\) centered at \(a\) of radius \(r\in(0,\infty)\),
$V_{\theta; K}$ is continuous at $a$ if and only if $V_{\theta;K\cap \bar{B}(a,r)}$ is continuous at \(a\). 

By Lemma~\ref{lem:weight-vs-unweight}-(ii), $V_{K,{\phi}/{c}}=c^{-1}(V_{\theta;K} + \phi)$, 
so $V_{\theta;K}$ is continuous if and only if $V_{K,\phi/c}$ is continuous. 
Consequently, 
$V_{\theta;K}$ is continuous.
Then for each \(a\in K\) and small \(r>0\) and closed holomorphic coordinate ball \(\bar{B}(a,r)\), $V_{\theta;K\cap \bar{B}(a,r)}$ is continuous at $a$. This enables us to use the same proof used for Theorem~\ref{thm:characterization-c}-(ii) to get the continuity of $V_{\theta;K, -\phi}$. Also, \(c\omega=\theta+dd^c(-\phi)\) gives
$V_{\theta;K,-\phi} = V_{c\omega; K} -\phi$. 
Therefore, $V_{c\omega;K}$ is also continuous.
By Remark~\ref{rmk:regularity-c}, $V_{K}=V_{\omega;K}$ is continuous.
\end{proof}

In fact, in Corollary~\ref{cor:loc-w-c}, if the weight function \(\psi:X\to\mathbb{R}\) is smooth, the last paragraph of its proof is enough for the continuity of \(V_K\) since \(dd^c\psi+c''\omega\) becomes another Hermitian metric on \(X\).

\begin{proof}[Proof of Theorem~\ref{thm:characterization-c-c}] 
(ii) implies (iii) by the definitions.

To show that (i) implies (ii), suppose \(V_K\) is continuous at \(a\in K\). Fix 
\(R\in(0,\infty)\) and 
a holomorphic coordinate ball \((\Omega,\tau)\) in \(X\) centered at \(a\) of radius \(R\). We can take a non-negative bounded smooth strictly psh function \(\rho_1\) on \(\Omega_s:=\tau^{-1}(\mathbb{B}(\mathbf{0},2R/3))\) with \(dd^c\rho_1\leq \omega\), \(\rho_1(a)=0\), \(\liminf_{x\to\partial \Omega_s}\rho_1(x)>0\) as \(\Omega_s\Subset\Omega\).  Let \(0<r\leq R/3\). By Lemma~\ref{lem:equivalence-notions}, since \(K\cap\tau^{-1}(\bar{\mathbb{B}}(\mathbf{0},r))\subset\Omega_s\), 
for the constant \(m:=\liminf_{x\to\partial\Omega_s}\rho_1(x)/{(1+\|\rho_1\|_{\Omega_s})}\in(0,\infty)\) where \(\|\cdot\|_{\Omega_s}\) denotes the supremum norm on \(\Omega_s\), we have 
$$\begin{aligned}
0&\leq m u_{\tau(K\cap\tau^{-1}(\bar{\mathbb{B}}(\mathbf{0},r))),\rho_1\circ\tau^{-1};
\tau(\Omega_s)}(z)\\
&\leq(V_{K\cap\tau^{-1}(\bar{\mathbb{B}}(\mathbf{0},r))}+\rho_1)\circ \tau^{-1}(z),\quad z\in
\tau(\Omega_s).\end{aligned}$$
The set 
\(\tau(\Omega_s)\) 
is the domain of \(u_{\tau(K\cap\tau^{-1}(\bar{\mathbb{B}}(\mathbf{0},r))),\rho_1\circ\tau^{-1};
\tau(\Omega_s)}\). 
By the inequality~\eqref{eq:re-ext-compare-c}, 
$$\begin{aligned}0&\leq u_{\tau(K\cap\tau^{-1}(\bar{\mathbb{B}}(\mathbf{0},r)));
\tau(\Omega_s)}(z)+
\inf_{\Omega_s}\rho_1\\
&\leq u_{\tau(K\cap\tau^{-1}(\bar{\mathbb{B}}(\mathbf{0},r))),\rho_1\circ\tau^{-1};
\tau(\Omega_s)}(z),\quad z\in\tau(\Omega_s).\end{aligned}$$

Our assumption and (i) give \(V_{K\cap\tau^{-1}(\bar{\mathbb{B}}(\mathbf{0},r))}^*(a)=0\). \(\rho_1^*(a)=0\). Then putting \(z=\mathbf{0}\) into the upper semicontinuous regularization of the first inequality gives $$u_{\tau(K\cap\tau^{-1}(\bar{\mathbb{B}}(\mathbf{0},r))),\rho_1\circ\tau^{-1};
\tau(\Omega_s)}^*(\mathbf{0})=0.$$
\(\rho_1(\Omega_s)\) has infimum \(0\), thus
putting \(z=\mathbf{0}\) to the regularization of the second inequality gives \(u_{\tau(K\cap\tau^{-1}(\bar{\mathbb{B}}(\mathbf{0},r)));
\tau(\Omega_s)}^*(\mathbf{0})=0\). 
\(V_{K\cap\tau^{-1}(\bar{\mathbb{B}}(\mathbf{0},r))}^*(a)=0\), so
\(
\tau(K\cap\tau^{-1}(\bar{\mathbb{B}}(\mathbf{0},r)))\) is not pluripolar by Proposition~\ref{prop:pluripolarextremal}. 
Accordingly,  
\(L_{\tau(K\cap\tau^{-1}(\bar{\mathbb{B}}(\mathbf{0},r)))}^*\) belongs to \(\mathcal{L}\) and  
\(\sup_{\tau(\Omega_s)}L_{\tau(K\cap\tau^{-1}(\bar{\mathbb{B}}(\mathbf{0},r)))}^*\) is finite.
Then by the inequality~\eqref{eq:compare-C}, 
$$\begin{aligned}0&\leq 
L_{\tau(K\cap\tau^{-1}(\bar{\mathbb{B}}(\mathbf{0},r)))}^*(\mathbf{0})\\&\leq
(\sup_{\tau(\Omega_s)}L_{\tau(K\cap\tau^{-1}(\bar{\mathbb{B}}(\mathbf{0},r)))}^*)\times
u_{\tau(K\cap\tau^{-1}(\bar{\mathbb{B}}(\mathbf{0},r)));
\tau(\Omega_s)}^*(\mathbf{0})\\&=0.\end{aligned}$$
Thus \(L_{\tau(K\cap\tau^{-1}(\bar{\mathbb{B}}(\mathbf{0},r)))}\) is continuous at \(\mathbf{0}\). Since \(r\in(0,R/3]\) was arbitrary, \(\tau(K\cap\tau^{-1}(\bar{\mathbb{B}}(\mathbf{0},R/2)))\) is locally \(L\)-regular at \(\mathbf{0}\).
Also, \(R\in(0,\infty)\) and \((\Omega,\tau)\) were arbitrary. Therefore, by Remark~\ref{rmk:loc-L-reg-mfd},
\(K\) is locally \(L\)-regular at \(a\).

To show that (iii) implies (i), assume \(K\) is weakly locally \(L\)-regular at \(a\). By Remark~\ref{rmk:loc-L-reg-mfd}, for some holomorphic coordinate ball \((\Omega_0,\tau_0)\) in \(X\) centered at \(a\) of finite radius \(R_0\), the set \(\tau_0(K\cap\tau_0^{-1}(\bar{\mathbb{B}}(\mathbf{0},R_0/2)))\) is locally \(L\)-regular at \(\mathbf{0}\).
Take a non-negative bounded smooth strictly psh function \(\rho_2\) on \(\Omega_{0s}:=\tau_0^{-1}(\mathbb{B}(\mathbf{0},2R_0/3))\) with \(\omega\leq dd^c\rho_2\) and \(\rho_2(a)=0\). 
Let \(\Omega_{0ss}:=\tau_0^{-1}(\mathbb{B}(\mathbf{0},R_0/2))\) and \(K_0:=K\cap\overline{\Omega_{0ss}}\). 
By the assumption, \(\tau_0(K_0)\) is \(L\)-regular at \(\mathbf{0}\), or equivalently, 
\(L_{\tau_0(K_0)}^*(\mathbf{0})=0\). 
Thus \(
K_0\) is not pluripolar, and then 
\(\|V_{K_0}\|_{\Omega_{0s}}\) is finite by Proposition~\ref{prop:pluripolarextremal}.
By Lemma~\ref{lem:equivalence-notions}, for \(\ M:=\|V_{K_0}\|_{\Omega_{0s}}+\|\rho_2\|_{\Omega_{0s}}+1<\infty\), we have 
$$0\leq(V_{K_0}+\rho_2)\circ\tau_0^{-1}(z)\leq Mu_{\tau_0(K_0),\rho_2\circ\tau_0^{-1};
\tau_0(\Omega_{0s})}(z), \quad z\in
\tau_0(\Omega_{0s}).$$
Let \(0<r\leq R_0/3\). By the monotonicity~\eqref{eq:re-ext-monotonicity} and inequality~\eqref{eq:re-ext-compare-c},
for \(z\in \tau_0(\Omega_{0s})\), 
$$\begin{aligned}
0&\leq u_{\tau_0(K_0),\rho_2\circ\tau_0^{-1};
\tau_0(\Omega_{0s})}(z)\\
&\leq u_{\tau_0(K_0)\cap\bar{\mathbb{B}}(\mathbf{0},r),\rho_2\circ\tau_0^{-1};
\tau_0(\Omega_{0s})}(z)\\
&\leq (1+\|\rho_2\|_{\Omega_{0s}})u_{\tau_0(K_0)\cap\bar{\mathbb{B}}(\mathbf{0},r);
\tau_0(\Omega_{0s})}(z)+\sup(\rho_2(K_0\cap\tau_0^{-1}(\bar{\mathbb{B}}(\mathbf{0},r)))).
\end{aligned}$$

By assumption,
\(L_{\tau_0(K_0)\cap\bar{\mathbb{B}}(\mathbf{0},t)}^*(\mathbf{0})=0\), \(0<t\leq R_0/3\). Then \(\tau_0(K_0)\cap\bar{\mathbb{B}}(\mathbf{0},r)\) is not pluripolar, and the inequality of comparability~\eqref{eq:compare-C} gives
$$0\leq
u_{\tau_0(K_0)\cap\bar{\mathbb{B}}(\mathbf{0},r);
\tau_0(\Omega_{0s})}^*(\mathbf{0})\leq\frac{L_{\tau_0(K_0)\cap\bar{\mathbb{B}}(\mathbf{0},r)}^*(\mathbf{0})}{\inf_{\partial(\tau_0(\Omega_{0s}))}L_{\tau_0(K_0)\cap\bar{\mathbb{B}}(\mathbf{0},r)}}=0.$$
$$\Big{(}\begin{aligned}&L_{\tau_0(K_0)\cap\bar{\mathbb{B}}(\mathbf{0},r)}(y)\geq L_{\bar{\mathbb{B}}(\mathbf{0},r)}(y)=\log^+{\frac{\|y\|}{r}},\\
&\inf_{\partial(\tau_0(\Omega_{0s}))} L_{\tau_0(K_0)\cap\bar{\mathbb{B}}(\mathbf{0},r)}\geq \log{2}>0.\end{aligned}\Big{)}$$
Putting \(z=\mathbf{0}\) into the regularization of the second three-line inequalities gives $$0\leq u_{\tau_0(K_0),\rho_2\circ\tau_0^{-1};
\tau_0(\Omega_{0s})}^*(\mathbf{0})\leq\sup(\rho_2(K\cap\tau_0^{-1}(\bar{\mathbb{B}}(\mathbf{0},r)))).$$ Since \(r\in(0,R_0/3]\) was arbitrary, letting \(r\to0^+\) yields $$0\leq u_{\tau_0(K_0),\rho_2\circ\tau_0^{-1};
\tau_0(\Omega_{0s})}^*(\mathbf{0})\leq\rho_2(a)=0.$$ Then putting \(z=\mathbf{0}\) into the regularization of the first inequality gives \(V_{K_0}^*(a)=0\). Therefore, 
\(V_{K}^*(a)=0\), and we have proved the continuity of \(V_K\) at \(a\in K\).
\end{proof}

\begin{cor}\label{cor:c-ext-inv-bihol}
The continuity of the extremal function of a compact set at a point in that compact set is invariant under any biholomorphism between two open sets of two compact Hermitian manifolds whose domain contains the point. 

In particular, the continuity of the extremal function of a compact set is invariant under any local biholomorphism between two open sets of two compact Hermitian manifolds whose domain contains the compact set. 
\end{cor}

\begin{proof}
Let \((Y,\omega_Y)\), \((Z,\omega_Z)\) be two compact Hermitian manifolds and let \(T\subset Y\) be a compact subset. 
Let \(b\in T\).
Let \(U_Y\) be an open neighborhood of \(b\) in \(Y\) and let \(U_Z\) be an open subset of \(Z\). Let \(\Phi:U_Y\mapsto U_Z\) be a biholomorphic map. 


Suppose the extremal function \(V_{\omega_Y;T}\) is continuous at \(b\). 
By Theorem~\ref{thm:characterization-c}-(iii), \(T\) is locally \(L\)-regular at \(b\).
Then for some holomorphic coordinate ball \((\Omega_Z,\tau_Z)\) in \(Z\) centered at \(\Phi(b)\in Z\) of small radius with
\(\Omega_Z\subset U_Z\),
\((\Phi^{-1}(\Omega_Z),\tau_{Z}\circ\Phi)\) is a holomorphic coordinate ball in \(Y\) centered at \(b\in T\), so we get local \(L\)-regularity of
$$(\tau_Z\circ\Phi)(T\cap\Phi^{-1}(\Omega_Z))=\tau_Z(\Phi(T\cap U_Y)\cap\Omega_Z)\quad\text{at}\quad\tau_Z(\Phi(b))=\mathbf{0}\in\mathbb{C}^n.$$ 
There exists some compact neighborhood \(S_b\subset U_Y\) of \(b\). Then \(\tau_Z(\Phi(T\cap S_b)\cap\Omega_Z)\) is also locally \(L\)-regular at \(\mathbf{0}\) since its intersection with \(\bar{\mathbb{B}}(\mathbf{0},r)\) is equal to the intersection of \(\tau_Z(\Phi(T\cap U_Y)\cap\Omega_Z)\) and \(\bar{\mathbb{B}}(\mathbf{0},r)\) for each small \(r>0\). 

By Theorem~\ref{thm:characterization-c}-(iii), weak local \(L\)-regularity at a point in a compact set is equivalent to continuity of extremal function of the set at the point. Thus 
\(V_{\omega_Z;\Phi(T\cap S_b)}\) 
is continuous at \(\Phi(b)\).
By monotonicity, 
\(V_{\omega_Z;\Phi(K\cap U_Y)}\) 
is continuous at \(\Phi(b)\).
\end{proof}

\begin{expl}
Let \(\mathbb{CP}^n\) be the complex projective space of complex dimension \(n\). Let \(K:=\{[z_0:\cdots:z_n]\in\mathbb{CP}^n:\sum_{1\leq j\leq n}|z_j|\leq |z_0|\}\) be a compact subset of \(\mathbb{CP}^n\). Let \((U_j,f_j)\) be the standard coordinate charts for \(\mathbb{CP}^n\) as $$U_j:=\{z_j\neq 0\},\quad f_j([z_0:\cdots:z_n]):=(\frac{z_0}{z_j},\dots,\frac{z_{j-1}}{z_j},\frac{z_{j+1}}{z_j},\cdots,\frac{z_n}{z_j}).$$
Let \(F\) be either the set \(\{w\in\mathbb{C}^n:\sum_{l}|w_l|\leq1\}\) or the set \(\{w\in\mathbb{C}^n:1+\sum_{2\leq l}|w_l|\leq|w_1|\}\). 
For each point \(a\in F\) and for each open neighborhood \(O\) of \(a\), there exists an invertible affine automorphism \(\Phi\) 
of \(\mathbb{C}^n\) such that \(\Phi([0,1]^n)\) is contained in \(F\cap O\) and has \(a\) as its vertex. Since the extremal function 
of \([0,1]^n\) is continuous by \cite[Corollary~5.4.5]{Kl91}, the extremal function 
of \(\Phi([0,1]^n)\) is continuous by the relation
$$L_{\Phi([0,1]^n)}=L_{[0,1]^n}\circ\Phi^{-1}.$$
Therefore, for each \(0\leq j\leq n\), \(f_j(K\cap U_j)\) is locally \(L\)-regular. This means that \(K\) is weakly locally \(L\)-regular. By Theorem~\ref{thm:characterization-c}, \(V_{\omega;K}\) is continuous for each Hermitian metric \(\omega\) on \(\mathbb{CP}^n\). 
\end{expl}

\begin{proof}[Proof of Theorem~\ref{thm:L-reg at starcenter implies loc L-reg}] 
Since \(b\in F\) is a star center of \(F\), for each \(t\in (0,1]\), the set \(b+(F-b)/t\) contains \(F\). Since \(F\) is compact, there exists \(\epsilon\in(0,1)\) such that \(\bar{\mathbb{B}}(b,1/\epsilon)\) contains \(F\). Then for each \(0<r\leq1\), we have
$$L_{F\cap\bar{\mathbb{B}}(b,r)}(z)= L_{(b+\frac{F-b}{r\epsilon})\cap\bar{\mathbb{B}}(b,\frac{1}{\epsilon})}(b+\frac{z-b}{r\epsilon})\leq L_F(b+\frac{z-b}{r\epsilon}), \quad z\in\mathbb{C}^n.$$
The \(L\)-regularity of \(F\) at \(b\) means \(L_F^*(b)=0\). Therefore, for each \(r\in (0,1]\), we have \(L_{F\cap\bar{\mathbb{B}}(b,r)}^*(b)=0\). In other words, \(F\) is locally \(L\)-regular at \(b\). 

If \(F\) is convex, then every point in \(F\) is a star center of \(F\). 
\end{proof}

\section{Appendix}
\label{sec:appendix}
In this section, we give the relation between extremal functions on \(X\) and locally defined weighted relative extremal functions in \(\mathbb{C}^n\). Nguyen in \cite{Ng24} proved this relation for a compact K\"ahler manifold, and we adjust his proof for a compact Hermitian manifold. 

Let \(a\in X\). There is a holomorphic coordinate ball $(\Omega, \tau)$ in \(X\) centered at \(a\), which is the relatively compact restricted chart in another holomorphic chart (if that bigger chart is \((U,f)\), \(U\) is open in \(X\), \(f\) is biholomorphic from \(U\) onto an open set in \(\mathbb{C}^n\), \(\Omega\Subset U,f|_\Omega=\tau\)): \(\Omega\) is open in \(X\), \(\tau\) is biholomorphic with
\[\label{eq:c-ball}\tau: \Omega \to \mathbb{B}:=\mathbb{B}(\mathbf{0},1) \subset \mathbb{C}^n, \quad \tau(\Omega)=\mathbb{B},\quad \tau(a)=\mathbf{0}.\] 
There exist non-negative bounded smooth strictly psh functions $\rho_1,\rho_2\in PSH(\Omega) \cap C^\infty(\Omega)\cap L^{\infty}(\Omega)$ 
such that \(\liminf_{x\to\partial\Omega}\rho_1(x)>0\), 
\(\rho_1(a)=\rho_2(a)=0\) and 
$dd^c\rho_1\leq\omega \leq dd^c \rho_2$.
(\(\rho_j(x):=A_j|\tau(x)|^2,\;x\in\Omega\), for some \(A_j\in(0,\infty)\) are such functions, as \(f|_{\Omega}=\tau\), \(\Omega\Subset U\).)  

For $E\Subset \mathbb{B}$, the (zero-one) relative extremal function of \(E\) on \(\mathbb{B}\) is 
\[\label{eq:zero-one-C}\notag
	u_{E}(z) =u_{E;\mathbb{B}}(z):= \sup \left\{ v(z) : v\in PSH(\mathbb{B}), v|_E \leq 0, v\leq  1 \right\},\quad
z\in\mathbb{B}.\]
\cite[Proposition~5.3.3]{Kl91} gives a relation of \(L_F\) and \(u_F\) for a compact non-pluripolar set \(F\) in \(\mathbb{B}\) with \(C_1:=\inf_{\partial\mathbb{B}} L_{F}\in (0,\infty)\) and \(C_2:=\sup_{\mathbb{B}} L_F^*=\sup_{\mathbb{B}}L_F\in (0,\infty)\) as
\[\label{eq:compare-C}
	C_1 \, u_F(z) \leq L_F(z) \leq C_2 u_F(z), \quad{z\in\mathbb{B}}.\]
For a bounded function \(\phi:\mathbb{B}\to\mathbb{R}\),
the weighted relative extremal function of \((E,\phi)\) on \(\mathbb{B}\) is
\[\label{eq:zero-one-CW}\notag
	u_{E,\phi} (z) =u_{E,\phi;\mathbb{B}}(z):=\sup \{ v(z) : v\in PSH(\mathbb{B}), \, v|_E \leq \phi|_E, v\leq  \phi+ 1 \},\quad z\in\mathbb{B}.\]
Like Lemma~\ref{lem:weight-vs-unweight}-(i),
by the definitions, (\(\|\cdot\|_{\mathbb{B}}\) denotes the supremum norm on \(\mathbb{B}\))
\[\label{eq:re-ext-compare-c}
	u_{E} +\inf_{\mathbb{B}}\phi \leq u_{E,\phi} \leq  (1+\|\phi\|_{\mathbb{B}})u_E + \sup_E \phi  \quad\text{when } \phi\geq0.\]
Like \eqref{eq:monotonicity}, for $E_1\subset E_2\Subset \mathbb{B}$ and $-\infty<-\|\phi_1\|_{\mathbb{B}}\leq \phi_1 \leq \phi_2\leq\|\phi_2\|_{\mathbb{B}}<\infty$, 
\[\label{eq:re-ext-monotonicity}
	u_{E_2,\phi_1} \leq u_{E_1,\phi_1} \leq u_{E_1, \phi_2}.\]

The following lemma tells a relation between an extremal function on \(X\) and the pullback of a weighted relative extremal functions on \(\mathbb{B}\) by a chart on \(X\). 
\begin{lem} \label{lem:equivalence-notions} Let $\emptyset\neq K$ be a compact non-pluripolar subset of \(X\) and \(a\in K\). Let \((\Omega,\tau)\) be the holomorphic chart centered at \(a\) given in \eqref{eq:c-ball} and \(\rho_1, \rho_2\) be the functions below \eqref{eq:c-ball}. If \(K\subset{\Omega}\), then there exist constants 
$0<M=M(K,\Omega,\|\rho_2\|_{\Omega})$ and 
$0<m=m(\liminf_{x\to\partial\Omega}\rho_1(x),\|\rho_1\|_{\Omega})$ 
such that,
for $\hat{\rho_j}:= \rho_j\circ \tau^{-1}$,
$$
 V_K\circ\tau^{-1}(z)+\hat{\rho_2}(z)
 \leq M\, u_{\tau(K),\hat{\rho_2}} (z), \quad z\in\mathbb{B},
$$
$$
 m\, u_{\tau(K),\hat{\rho_1}} (z) \leq 
 V_K\circ\tau^{-1}(z)+\hat{\rho_1}(z),\quad z\in\mathbb{B}.
$$
In fact, the second inequality holds even when \(K\) is pluripolar.
\end{lem}
\begin{proof} We use the proof of \eqref{eq:compare-C} in \cite[Proposition~5.3.3]{Kl91}. 
$\sup_\Omega V_K^*=\|V_K^*\|_{\Omega} <\infty$ by Proposition~\ref{prop:pluripolarextremal} as \(K\) is not pluripolar. 
Let $v\in PSH(X,\omega)$, $v|_K\leq 0$. 
Then $v +\rho_2$ is in $ PSH({\Omega})$. Take $M:=\|V_K^*\|_{\Omega} +\|\rho_2\|_{\Omega}+1$. 
Since $M\geq1$ and $\rho_2 \geq 0$, 
$$(v +\rho_2)|_{K} \leq {\rho_2}|_K\leq M\rho_2|_K, 
\quad v+\rho_2\leq M\leq M(1+\rho_2).$$ 
These mean \(v+\rho_2\leq M u_{\tau(K),\hat{\rho_2}}\circ \tau\). Since \(v\) was an arbitrary competitor for \(V_K\), 
we get the first inequality.

To obtain the second inequality,
Let $m':=\liminf_{x\to\partial\Omega} \rho_1(x)>0$. 
Take an open neighborhood \(O\) of \(K\). Since \(O\) is not pluripolar, \(V_O=V_O^*\in PSH(X,\omega)\) by Propositions~\ref{prop:pluripolarextremal}, \ref{prop:elementary}. 
Let $v \in PSH(\mathbb{B})$ be a competitor for \(u_{\tau(K),\hat{\rho_1}}\). In other words, $v|_{\tau(K)} \leq \hat{\rho_1}|_{\tau(K)}$, $v\leq \hat{\rho_1}+1$. Define 
$$v_O|_\Omega:=\max\{\frac{m'}{1+\|\rho_1\|_{\Omega}}v\circ\tau-\rho_1,V_O\}|_\Omega,\quad
v_O|_{X\setminus\Omega}:=V_O|_{X\setminus\Omega}.$$ 
By Proposition~\ref{prop:envelope}-(i), as
$$\limsup_{x\to \partial\Omega} 
[\frac{m'
}{1+\|\rho_1\|_{\Omega}} v\circ \tau (x)
- \rho_1(x)] \leq 
m'- \liminf_{x\to\partial \Omega} \rho_1(x)
=0\leq V_{O},$$ 
we have \(v_O\in PSH(X,\omega)\).
Let $m:=m'/(1+\|\rho_1\|_\Omega)$. 
Since \(m\in (0,1)\), $v\circ\tau|_K\leq\rho_1|_K$ and $\rho_1\geq0$, we have $(mv\circ\tau-\rho_1)|_K\leq0$. Also, \(V_O=0\) on \(K\). Therefore, \(v_O\leq V_K\), \(mv\circ\tau-\rho_1\leq V_K|_\Omega\). 
Since \(v\) was an arbitrary competitor for \(u_{\tau(K),\hat{\rho_1}}\), the second inequality holds.
\end{proof}

\end{document}